\documentclass[10pt]{article}
\usepackage{amsmath}
\usepackage{amsfonts}
\usepackage{amssymb}
\usepackage{graphicx}
\usepackage{amsthm}
\usepackage{hyperref}
\usepackage{xcolor}
\usepackage{accents}
\usepackage{graphics,wrapfig, graphicx, mathrsfs,xcolor}
\usepackage{mathtools}
\usepackage{amssymb,amscd,graphicx,xcolor,subfig,tikz}
\usepackage{graphics}
\usepackage{indentfirst}
\usepackage{bm, enumerate,xcolor}
\usepackage[dvips]{epsfig}
\usepackage{latexsym}
\usepackage{float}
\usepackage{amsmath}
\usepackage{amssymb} 
\usepackage[numbers,sort&compress,square]{natbib}

\usepackage[english]{babel}
\usepackage [autostyle, english = american]{csquotes}
\MakeOuterQuote{"}

\numberwithin{equation}{section}
\theoremstyle{plain}
\newtheorem*{theorem*}{Theorem}
\newtheorem*{lemma*}{Lemma}
\newtheorem{theorem}{Theorem}
\newtheorem{lemma}{Lemma}[section]
\newtheorem{corollary}[lemma]{Corollary}

\newtheorem{proposition}[lemma]{Proposition}

\theoremstyle{definition}

\newtheorem{remark}[lemma]{Remark}

\newtheorem*{remark*}{Remark}
\newtheorem*{example*}{Example}
\newtheorem*{er*}{Examples and Remarks}

\usepackage{geometry}
\begin{document}

\title{On the Loss and Propagation of Modulus of Continuity for the Two-Dimensional Incompressible Euler Equations }
\author{Karim R. Shikh Khalil}
\maketitle
\abstract{
It is known from the work of Koch that the two-dimensional incompressible Euler equations propagate Dini modulus of continuity for the vorticity. In this work, we consider the two-dimensional Euler equations with a modulus of continuity for vorticity rougher than Dini continuous. We first show that the two-dimensional Euler equations propagate an explicit family of moduli of continuity for the vorticity that are rougher than Dini continuity. The main goal of this work is to address the following question: Given a modulus of continuity for the 2D Euler equations, can we always propagate it? The answer to this question is \textit{No}. We construct a family of moduli of continuity for the 2D Euler equations that are \textit{not} propagated.

}

\section{Introduction}

The Euler equations
\begin{equation}\label{Euler}
\begin{split}
\partial_t u &+  u\cdot\nabla u + \nabla p = 0 \\ 
&\nabla \cdot u = 0 \\  
\end{split}
\end{equation}
describe the motion of perfect incompressible fluids, where $u$ is the velocity field and $p$ is the pressure field that enforces the incompressibility condition. The challenge of studying the Euler equations arises from the fact that the equations are nonlinear and nonlocal.

 To study the Euler equations, it is useful to introduce the vorticity \(\omega := \nabla^\perp \cdot u\). The 2D Euler equations can then be written as follows:
\begin{equation}\label{EulerVorticty}
\begin{split}
\partial_t \omega &+  u\cdot\nabla \omega= 0, \\ 
&\nabla \cdot u=0 \\  
&u=\nabla^\perp \Delta^{-1} \omega.\\
\end{split}
\end{equation}

In this work, we are interested in studying the propagation of the modulus of continuity of the vorticity for the two-dimensional Euler equations. It is known from the work of Yudovich \cite{Y1} that the 2D Euler equations are globally well-posed for bounded vorticity with sufficient decay, allowing one to study the long-time behavior of solutions. The 2D Euler dynamics depend significantly on the choice of space in which the dynamics are studied. Depending on the choice of space, it could be too fine to capture large-scale structures, leading to instabilities. For instance, when considering the 2D Euler dynamics in H\"older  $C^{\alpha}$ space, the $C^{\alpha}$ norm of vorticity has a double exponential upper bound. In fact, it is known from the work of Kiselev-\v{S}ver\'ak  \cite{KS} that the gradient of vorticity grows double exponentially on the disk domain. See the following works for recent progress on the infinite time growth of the gradient of vorticity: \cite{CJ, D, D2, DE, DEJ, N, K, KS, Y3, Y4, Z}.

 In order to try to capture stable behavior, one can study the 2D Euler equations in a rougher but still continuous function spaces. In the work of Koch \cite{K}, the author considers (among other important results) vorticity with Dini modulus of continuity. Namely,  a modulus of continuity  $\Omega(\rho)$ that satisfies the following:   
  $$
 \int_0^{\frac{1}{2}} \frac{\Omega(\rho)}{\rho}  d \rho < \infty.
 $$ 

In \cite{K}, Koch proves that the Dini modulus of continuity for the 2D Euler equations can be propagated with exponential bounds, and thus exhibits slower growth compared to the H\"older $C^{\alpha}$ space, where the $C^{\alpha}$  norm has a double exponential bound. Examples of Dini continuous functions include \(\Omega(\rho) = |\log(\rho)|^{-\gamma}\), for \(\gamma > 1\).

In this work, we study the 2D Euler equations with a rougher modulus of continuity than Dini continuous. We observe that if we consider an explicit modulus of continuity of the form $\Omega(\rho) = |\log(\rho)|^{-\gamma}$ for $0 < \gamma < 1$, it is not difficult to see that we can actually propagate it in time (see Section \ref{PropEx}). We observe that these explicit moduli of continuity grow with an exponential bound of the form $e^{\gamma t}$. In other words, the rougher these moduli of continuity become (smaller $\gamma$), the slower the growth.

The main goal of this work is to address the following question: Given a modulus of continuity for the 2D Euler equations, can we always propagate it? The answer to this question is \textit{No}. We construct a family of moduli of continuity for the 2D Euler equations that are \textit{not} propagated.

Finally, we should remark that in this work, we are considering continuous solutions and we are interested in the propagation of the modulus of continuity for vorticity. In other words, the solutions we consider are bounded. For unbounded vorticity, there has been some recent work on the propagation of different structures, see \cite{CCS, DEL, EMR, L}. Lastly, we remark that the literature on the well/ill-posedness of 2D Euler equations is extensive, and we do not attempt to be exhaustive here. For well-posedness, see the following  works \cite{Li, G, W, H, EbM, Ka, BKM}. For ill-posedness and loss of regularity, see the following  works \cite{BL1, BL2, CMO, EM, EJ1, J}.

\subsection{Statement of the main result}

\begin{theorem}
There exist a family of solutions to the 2D Euler equation for which the modulus of continuity for the vorticity is lost. Namely, there exist solutions to 2D Euler equations with initial  modulus of continuity $\Omega_0(\rho)$ for the vorticity such that for $t>0$, we have

$$\frac{\Omega(t,\rho)}{\Omega_0(\rho)} \geq \frac{1}{2} \log(|\log(\rho)|)    $$

for $\rho>0$ arbitrarily small. 

\end{theorem}

\subsection{A short sketch of the proof}\label{sketch}
The idea of the proof is as follows: Due to Koch's result \cite{K}, we know that the modulus of continuity needs to be rougher than Dini continuous. Additionally, we know that if we considered a modulus of continuity of the following form: \(\Omega(\rho) = |\log(\rho)|^{-\gamma}\), for \(0 < \gamma < 1\), it will be propagated (see Section \ref{PropEx}). The idea is to construct a modulus of continuity that alternates, over specific spatial scales, between the following two moduli of continuity:
\[ 
\log(|\log(\rho)|) \cdot |\log(\rho)|^{-\gamma} \quad \text{and} \quad |\log(\rho)|^{-\gamma},
\]
while preserving the concavity and monotonicity of the modulus of continuity.
 Here we will give a rough sketch of why we consider such a construction. Given an appropriately chosen initial modulus of continuity \(\Omega_0(\cdot)\), the radial trajectory of the flow map for the 2D Euler equations heuristically follows the following ODE (see Section \ref{Loss} for the rigorous argument/justification):

$$
\frac{dx}{dt}\approx x \, \Omega_0(x_0) |\log(x_0)| \quad \text{where} \quad x|_{t=0}=x_0.
$$

Thus,  if we consider a modulus of continuity such that $$\Omega_0(x_0)\approx|\log(x_0)|^{-\gamma} \quad \text{for} \quad 0<\gamma<1.$$   Then,

$$
\frac{dx}{dt} \approx x  |\log(x_0)|^{1-\gamma}. 
$$

Now because the vorticity is transported,  modulus of continuity $\Omega$ at $t=O(1)$ will take the following form:

$$\Omega(t)\approx \Omega( x_0  |\log(x_0)|^{1-\gamma} ).$$

 We observe now that  if  we just considered a modulus of continuity of the form  $\Omega(x)=|\log(x)|^{-\gamma}$ then at $t=O(1)$, we have 

$$
\frac{\Omega(t)}{\Omega_0}\approx \frac{ |\log (x_0  |\log(x_0)|^{1-\gamma}  ))|^{-\gamma}}{|\log(x_0)|^{-\gamma}}\approx O(1).
$$

The idea is to construct the  modulus of continuity such that at the spatial scale $x\approx x_0  |\log(x_0)|^{1-\gamma}$, where the trajectory of the particle will arrive at time $t=O(1)$, the modulus of continuity takes the following form: 

$$|\log (x_0  |\log(x_0)|^{1-\gamma}  ))|^{-\gamma} \log|\log(x_0  |\log(x_0)|^{1-\gamma} )|, $$ such that we now have instead the following:

$$
\frac{\Omega(t)}{\Omega_0}\approx \frac{ |\log (x_0  |\log(x_0)|^{1-\gamma}  ))|^{-\gamma} \log|\log(x_0  |\log(x_0)|^{1-\gamma} )| }{|\log(x_0)|^{-\gamma}}\approx  \log|\log(x_0)|. 
$$

The construction will be done along a sequence \(x_n \rightarrow 0\) such that the modulus of continuity stays concave and monotonic. We see from above that the modulus of continuity will not be explicit. The main contribution of this work is the construction of the modulus of continuity, which will be done in Section \ref{ConsInit}. To control the 2D Euler dynamics and justify the above argument, we follow the strategy of Elgindi-Jeong \cite{EJ1} and Jeong \cite{J}. We especially benefited from the argument by Jeong in \cite{J}. The idea here is to use a generalization of the Key Lemma of Kiselev-\v{S}ver\'ak  in \cite{KS} by Elgindi \cite{E} and Elgindi-Jeong \cite{EJ2}, see Lemma \ref{EJLemma}, which will allow us to obtain sharp estimates for the flow map trajectories and control the dynamics. This will be done in Section \ref{Loss}.

\subsection{Organization: } 

This paper is organized as follow: In Section \ref{Prelim}, we include some preliminary results that will be used throughout the paper.  In Section \ref{PropEx}, we show the propagation of an explicit family of moduli of continuity for the 2D Euler equations, which are rougher than Dini continuous. In Section \ref{Const}, we construct initial data with the modulus of continuity that will be shown to be lost for the 2D Euler equations. In Section \ref{Loss}, we prove our main result, which shows that there exist solutions for the 2D Euler equations that lose their modulus of continuity.

\subsection{Notation} 
In this paper, we will be working with polar coordinates. Given a point $x=(x_1,x_2) \in \mathbb{R}^2 $ in Cartesian coordinates, we will write $$r=\sqrt{x_1^2+x_2^2}, $$ to denote the radial component, and 
$$\theta=\tan^{-1}\big(\frac{x_2}{x_1}\big), $$
to denote the angular component. We will interchangeably use the notation $x=(x_1,x_2)=(r,\theta)$ to refer to a point in $\mathbb{R}^{2}$. 

For the velocity field, We will use the notation $u=u^{r}e_r+u^{\theta}e_{\theta}$, where $u^{r} \,  \text{and} \, \,  u^{\theta}$  are the radial and angular components of the velocity field, respectively, and $e_r\,\, \text{and} \,  \, e_{\theta}$ are the unit vectors in the radial and angular directions, respectively

Further, given a point $(r_0,\theta_0)$ in polar coordinates, we denote the flow map $$\Phi(t,r_0,\theta_0)=(\Phi_{r}(t,r_0,\theta_0) , \Phi_{\theta}(t,r_0,\theta_0)),$$ where at time $t$, we have $\Phi_{r}(t,r_0,\theta_0) \, \text{and}   \, \Phi_{\theta}(t,r_0,\theta_0)$   denoting respectively the radial and angular locations of the particle  $(r_0,\theta_0)$. To shorten the notation, we will occasionally suppress the dependence on the initial point $(r_0,\theta_0)$ and write 
$\Phi(t)=(\Phi_{r}(t) , \Phi_{\theta}(t))$.

\section{Preliminary results:}\label{Prelim}
In this section, we include some preliminary results which we will use in the paper. In Lemma \ref{YudEst}, we recall the Log-Lipschitz estimates for the velocity field for bounded vorticity. In Lemma \ref{YudEstFlow}, we recall the Yudovich estimate for the flow map when the velocity field is Log-Lipschitz. Finally, in Lemma \ref{EJLemma}, we recall a generalization of the Key lemma of Kiselev-\v{S}ver\'ak  in \cite{KS} by Elgindi \cite{E} and Elgindi-Jeong \cite{EJ2}, which we use in Section \ref{Loss}.

\begin{lemma}\label{YudEst}(Yudovich \cite{Y1}) Let $\omega$ be a bounded vorticity with compact support, $\omega \in L^{\infty}_c(\mathbb{R}^2)$. Then the corresponding velocity field  $u=\nabla^{\perp}\Delta^{-1} \omega$ satisfies the following:

$$
|u(t,x)-u(t,y)| \leq c |\omega|_{L^{\infty} \cap L^{1}} |x-y| \,  |\log(|x-y|)|,
$$
whenever $|x-y|$ is small.

\end{lemma}

\begin{proof}

See textbooks \cite{MB,MP}, which are standard reference for the above result.

\end{proof}

\begin{lemma}\label{YudEstFlow}(Yudovich \cite{Y1})  Let $\omega$ be a bounded vorticity with compact support, $\omega \in L^{\infty}_c(\mathbb{R}^2)$. Then the flow map  $\Phi$  defined by:
$$\frac{d}{dt}\Phi=u(\Phi)$$
satisfies the following:
$$
c|x-y|^{e^{ct |\omega|_{L^{\infty} \cap L^{1}} }} \leq |\Phi(t,x)-\Phi(t,y)| \leq C|x-y|^{e^{-C t |\omega|_{L^{\infty} \cap L^{1}}}},
$$

whenever $|x-y|$ is small. 

\end{lemma}

\begin{proof}

For $x, y \in \mathbb{R}^2$, the flow map corresponding to particles starting at  $x, \text{and} \,  y$, respectively, is defined as follows:  $$\frac{d}{dt}\Phi(t,x)=u(t,\Phi(t,x)),  \, \text{and} \,   \, \frac{d}{dt}\Phi(t,y)=u(t,\Phi(t,y))   $$

with $\Phi(0,x)=x$, and $\Phi(0,y)=y$.   Thus, using the estimate on the velocity field from Lemma \ref{YudEst}, we obtain

$$|\omega|_{L^{\infty} \cap L^{1}}|\Phi(t,x)-\Phi(t,y)|\log(|\Phi(t,x)-\Phi(t,y)|) \leq \frac{d}{dt} |\Phi(t,x)-\Phi(t,y)|\leq  |\omega|_{L^{\infty} \cap L^{1}} |\Phi(t,x)-\Phi(t,y)|  \, |\log(|\Phi(t,x)-\Phi(t,y)|)|. $$

Then,  the statement follows from Gronwall inequality/comparison principle by the solving the corresponding ODE.

\end{proof}

The following Lemma \ref{EJLemma} from Elgindi \cite{E} and Elgindi-Jeong \cite{EJ2} allows us to precisely locate the log-Lipschitz part of the velocity field. Lemma \ref{EJLemma} is generalization of the Key Lemma of Kiselev-\v{S}ver\'ak in \cite{KS} for any bounded vorticity on the whole plane $\mathbb{R}^2$ with compact support. We will use this Lemma to obtain sharp estimates on the flow map in Section \ref{Loss}. See  the work of Elgindi \cite{TE} for the Biot-Savart decomposition for the 3D Euler equations.

\begin{lemma}\label{EJLemma}(Elgindi \cite{E}, Elgindi-Jeong \cite{EJ2} Lemma 5.1) Let $\omega$ be bounded vorticity with compact support, $ \omega  \in L^{\infty}_c(\mathbb{R}^2)$. Then the corresponding velocity field $u=\nabla^{\perp}\Delta^{-1} \omega$ satisfies the following:
   $$
\Big|u(r,\theta)-u(0)-\frac{1}{2\pi} \begin{pmatrix}
-\cos(\theta)\\
\sin(\theta)
\end{pmatrix} r I^s(r) -\frac{1}{2\pi} \begin{pmatrix}
\sin(\theta)\\
\cos(\theta)
\end{pmatrix} r I^c(r) \Big| \leq C  r |\omega|_{L^\infty}
 $$
where the constant $C$ is independent of the support of $\omega$, and the operators  $I^s$ and $I^c$ are defined as follows:
$$
I^s(r)=\int_r^{\infty} \int_0^{2\pi} \sin(2\theta)\frac{ \omega(s,\theta)}{s}  \, d \theta \, ds, \,  \, \text{and}  \quad I^c(r)=\int_r^{\infty} \int_0^{2\pi} \cos(2\theta)\frac{ \omega(s,\theta)}{s}  \, d \theta \, ds.
$$

\end{lemma}

\section{Propagation of explicit examples of moduli of continuity}\label{PropEx}

In this section, we show that the 2D Euler equations propagate an explicit family of moduli of continuity for the vorticity of the following form:
\begin{equation}\label{rmodulus}
\begin{split}
\Omega(\rho)= |\log(\rho)|^{-\gamma}, \,\, \,   \text{for}  \,\, 0 <\gamma< 1.
\end{split}
\end{equation}
These moduli of continuity are rougher than Dini modulus of continuity, which are known to be propagated by Euler equations, as shown by Koch \cite{K}. As mentioned, the key idea to show the propagation of Dini modulus of continuity for 2D Euler is the fact that when the vorticity is Dini continuous, the velocity field is Lipschitz.  The family of moduli of continuity we consider, \eqref{rmodulus}, allows for the gradient of the velocity field to be unbounded, but using the structure of the modulus of continuity allows us to still propagate it by applying Yudovich's estimate, which only requires the vorticity to be bounded. See also the work of Chae-Jeong in \cite{ChJ} where the authors show the preservation of the log-H\"older coefficients of the vorticity when the velocity field is Lipschitz.

\begin{proposition} Let $\omega(t)$ be a solution of 2D Euler equations with the following modulus of continuity for the initial data:

$$
|\omega_0(x)-\omega_0(y)|\leq \Omega_0(|x-y|)= \frac{1}{|\log( \, |x-y|)|^{\gamma}}, \,\, \,   \text{for}  \,\, 0 <\gamma< 1,
$$
for all $|x-y| \leq \frac{1}{2}$. Then the modulus of continuity is propagated in time.

\end{proposition}

\begin{proof} Here we have

$$  
\Omega(t,\rho)= \sup_{|x-y|\leq \rho}  |\omega(t,x)-\omega(t,y)| = \sup_{|x-y|\leq \rho} |\omega_0(\Phi^{-1}_t(x))-\omega_0(\Phi^{-1}_t(y))| \
$$
but we know that

$$|\omega_0(\Phi^{-1}_t(x))-\omega_0(\Phi^{-1}_t(y))| \leq \Omega_0(|\Phi_t^{-1}(x)-\Phi_t^{-1}(y)|) =  \frac{1}{|\log( \, |\Phi^{-1}_t(x)-\Phi^{-1}_t(y)|)|^{\gamma}}.   $$

Thus by Yudovich estimates,    which only requires $\omega$ to be in $L^{\infty}$, we have 

$$
\frac{1}{|\log( \, |\Phi^{-1}_t(x)-\Phi^{-1}_t(y)|)|^{\gamma }}   \leq  \frac{1}{|\log( \, |x-y|^{e^{-t}})|^{\gamma}}=\frac{e^{\gamma t}}{|\log( \, |x-y|)|^{\gamma}}.
$$

Hence,

$$|\omega_0(\Phi^{-1}_t(x))-\omega_0(\Phi^{-1}_t(y))| \leq  \frac{e^{\gamma t  }}{|\log( \, |x-y|)|^{\gamma }}. $$

Thus

$$\Omega(t,\rho) \leq \Omega_0(\rho)  e^{\gamma t }.  $$

\end{proof}

\begin{remark}
We observe that the rougher the modulus of continuity, smaller $\gamma $, the slower the growth in time.  
\end{remark}

\begin{remark}
The same above argument can also be used to show that the following family of moduli of continuity can also be propagated: 
\begin{equation*}\label{rmodulus2}
\begin{split}
\Omega(\rho)= |\log(\rho)|^{-\gamma} \log(|\log(\rho)|) \,\, \,   \text{for}  \,\, 0 <\gamma< 1.
\end{split}
\end{equation*}

\end{remark}

\section{Construction of initial data}\label{ConsInit}

In this section, we construct the main part of the initial data, which has a modulus of continuity that will be lost for the 2D Euler equations. We consider this to be the main contribution of this work. The function constructed here will form the radial part of the initial data; see Section \ref{Loss} for more details. This section is organized as follows: in Proposition \ref{Const}, we show how to construct the initial data. In Corollary \ref{ScaledConst}, we demonstrate how we can rescale the initial data for use in Section \ref{Loss}.

\begin{proposition}\label{Const}
  There exist  a family of  continuous monotone concave functions $G(x)$ on $[0,1]$ such that   
\begin{equation}\label{unbdd}
\begin{split}
\limsup_{x\rightarrow 0} \frac{G(x \,G(x) \, |\log(x)|)}{G(x)}=\infty.
\end{split}
\end{equation}
Quantitatively, we can find a sequence $x_n \rightarrow 0$, as $n\rightarrow \infty$, such that 

$$
\frac{G( x_n \,G(x_n) \, |\log(x_n)|)}{G(x_n)}\geq   \frac{1}{2}\log(|\log(x_n)|).
$$

\end{proposition}

\begin{proof}

We will construct the function iteratively. We will define the function \(G(x)\) on discrete points \(x_n\) for \(n \in \mathbb{Z}^{+}\), and then we will linearly interpolate between \(G(x_n)\) points to define \(G(x)\) everywhere else. For notation, given any two points \(a\) and \(b\), we define \(L_{a,b}(x)\) to be the equation of the line connecting them.

\begin{figure}[ht]
\centering
\tikzset{every picture/.style={line width=0.75pt}} 

\begin{tikzpicture}[x=0.75pt,y=0.75pt,yscale=-1,xscale=1]

\draw [line width=0.75]    (194.01,189.04) .. controls (194.01,94.08) and (285.01,33.04) .. (382.01,29.04) ;
\draw    (194.01,189.04) .. controls (194.01,100.04) and (284.01,47.04) .. (383.01,39.04) ;
\draw [line width=1.5]    (219.01,109.04) -- (357.01,31.04) ;
\draw [line width=1.5]    (203.85,141.35) -- (219.01,109.04) ;
\draw  (102.18,188.82) -- (466.75,189.71)(194.42,13.48) -- (193.88,232.92) (459.76,184.69) -- (466.75,189.71) -- (459.74,194.69) (189.4,20.47) -- (194.42,13.48) -- (199.4,20.49)  ;
\draw  [color={rgb, 255:red, 208; green, 2; blue, 27 }  ,draw opacity=1 ][line width=2.25] [line join = round][line cap = round] (204.12,141.05) .. controls (204.12,140.97) and (204.12,140.89) .. (204.12,140.8) ;
\draw  [color={rgb, 255:red, 208; green, 2; blue, 27 }  ,draw opacity=1 ][line width=2.25] [line join = round][line cap = round] (218.87,108.55) .. controls (218.87,108.55) and (218.87,108.55) .. (218.87,108.55) ;
\draw  [color={rgb, 255:red, 208; green, 2; blue, 27 }  ,draw opacity=1 ][line width=2.25] [line join = round][line cap = round] (324.46,49.26) .. controls (324.46,49.26) and (324.46,49.26) .. (324.46,49.26) ;
\draw  [color={rgb, 255:red, 208; green, 2; blue, 27 }  ,draw opacity=1 ][line width=2.25] [line join = round][line cap = round] (356.21,31.01) .. controls (356.21,31.01) and (356.21,31.01) .. (356.21,31.01) ;
\draw  [dash pattern={on 0.84pt off 2.51pt}]  (357.01,31.04) -- (356.71,189.89) ;
\draw  [dash pattern={on 0.84pt off 2.51pt}]  (324.21,49.44) -- (324.02,190.36) ;
\draw  [dash pattern={on 0.84pt off 2.51pt}]  (219.01,109.04) -- (219.23,189.56) ;
\draw  [dash pattern={on 0.84pt off 2.51pt}]  (203.85,141.35) -- (204.03,190.36) ;

\draw (348.82,192) node [anchor=north west][inner sep=0.75pt]  [font=\footnotesize]  {$x _{0}$};
\draw (316.02,191.4) node [anchor=north west][inner sep=0.75pt]  [font=\footnotesize]  {$x _{1}$};
\draw (212.22,191.4) node [anchor=north west][inner sep=0.75pt]  [font=\footnotesize]  {$x _{2}$};
\draw (195.99,191.44) node [anchor=north west][inner sep=0.75pt]  [font=\footnotesize]  {$x _{3}$};
\draw (387.01,16.44) node [anchor=north west][inner sep=0.75pt]  [font=\scriptsize]  {$|\log( x ) |^{-\gamma } \ \log( |\log( x ) |)$};
\draw (387.01,37.44) node [anchor=north west][inner sep=0.75pt]  [font=\scriptsize]  {$|\log( x ) |^{-\gamma }$};

\end{tikzpicture}
\caption{Sketch of the initial construction of the function $G(x)$.}\label{1-1}
\end{figure}
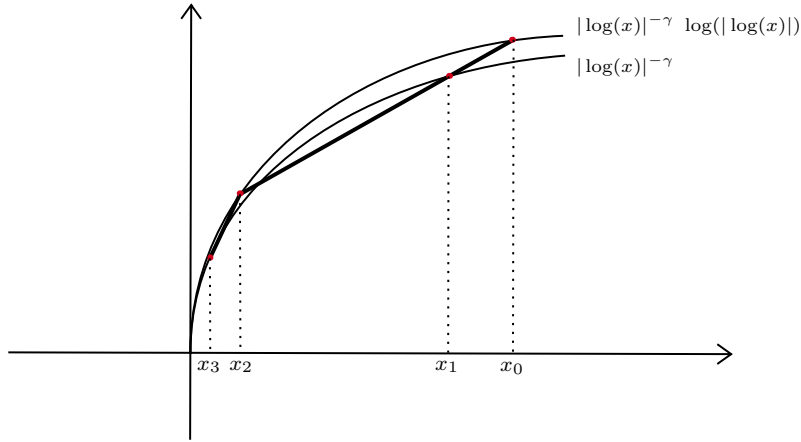

We take \(0 < \gamma < 1\), and we start by picking \(x_0\) small enough and defining  $$
G(x_0)=\frac{\log(|\log(x_0)|)}{|\log(x_0)|^{\gamma}}.
$$

Next, choose $x_1$ such that $x_1|\log(x_1)|^{1-\gamma}=x_0$, and define

$$
G(x_1)=\frac{1}{|\log(x_1)|^{\gamma}}.
$$

Now on \([x_1, x_0]\), we linearly interpolate to define \(G(x)\), see Figure 1. After that, we define \(x_2\) to be the intersection of the line connecting \(G(x_0)\) and \(G(x_1)\), \(L_{(G(x_0), G(x_1))}\), with the function \(f(x) = \frac{\log(|\log(x)|)}{|\log(x)|^{\gamma}}\) (which will always be well-defined, see Property 1 and Property 2 below), and define
$$G(x_2)=\frac{\log(|\log(x_2)|)}{|\log(x_2)|^{\gamma}}.$$

On \([x_2, x_1]\), we again linearly interpolate to define \(G(x)\). Finally, we define \(x_3\) such that \(x_3|\log(x_3)|^{1-\gamma} = x_2\), and then we just keep repeating the procedure. Namely, for an even \(n\), the point \(x_n\) is defined as the intersection of the line connecting \(G(x_{n-2})\) and \(G(x_{n-1})\) with the function \(f(x) = \frac{\log(|\log(x)|)}{|\log(x)|^{\gamma}}\), and then \(G(x_n)\) is defined as follows:

 $$G(x_n)=\frac{\log(|\log(x_n)|)}{|\log(x_n)|^{\gamma}}.$$ Then, the point $x_{n+1}$ is defined as  $x_{n+1}|\log(x_{n+1})|^{1-\gamma}=x_{n}$ and we set  $$
G(x_{n+1})=\frac{1}{|\log(x_{n+1})|^{\gamma}}.
$$

Now assuming that we can keep iterating this procedure as \(n \rightarrow \infty\) while keeping the function \(G(x)\) concave and vanishing at zero (this will be verified below, see Property 1 and Property 2), we first observe that

$$
\frac{G(x_0)}{G(x_1)} =\frac{G(x_1 \,G(x_1) \, |\log(x_1)|)}{G(x_1)}=\frac{G(x_1 |\log(x_1)|^{1-\gamma})}{G(x_1)} \geq  \log(|\log(x_1|\log(x_1)|^{1-\gamma})|).
$$

 Similarly, for  an even $n$, we have

$$
\frac{G(x_n)}{G(x_{n+1})}=\frac{G(x_{n+1} \,G(x_{n+1}) \, |\log(x_{n+1})|)}{G(x_{n+1})}=\frac{G(x_{n+1} |\log(x_{n+1})|^{1-\gamma})}{G(x_{n+1})} \geq \log(|\log(x_{n+1}|\log(x_{n+1})|^{1-\gamma})|).
$$

Thus,

$$
\frac{G(x_n)}{G(x_{n+1})}=\frac{G(x_{n+1} \,G(x_{n+1}) \, |\log(x_{n+1})|)}{G(x_{n+1})} \geq  \frac{1}{2}\log(|\log(x_n)|),
$$

and since we have $x_n \rightarrow 0$ as $n\rightarrow \infty$, we obtain

$$
\lim_{n \rightarrow \infty}\frac{G(x_n)}{G(x_{n+1})}= \lim_{n \rightarrow \infty} \frac{G(x_{n+1} \,G(x_{n+1}) \, |\log(x_{n+1})|)}{G(x_{n+1})}= \infty.
$$

Hence, we have $\eqref{unbdd}$. The only remaining step is to check that we can keep iterating the above procedure while maintaining the function \(G(x)\) to be concave and vanishing at zero. To see this, we need to verify \textbf{Property 1} and \textbf{Property 2} below. It is an elementary computation to verify these properties, but we include them here for completeness. It is enough to verify the properties for the construction of the function \(G(x)\) on the interval \([x_1, x_0]\) and show that \(x_2\) is well-defined because the procedure is iteratively the same.

\begin{itemize}

\item \textbf{Property 1):} The slope of the line \(L_{(G(x_0), G(x_1))}\) connecting \(G(x_0)\) and \(G(x_1)\) is less than the slope of the line \(L_{(G(x_0), 0)}\) connecting \(G(x_0)\) and zero.

We need this property because otherwise the function would not stay concave while still vanishing at zero. In addition, we need this property to ensure that the line \(L_{(G(x_0), G(x_1))}\) connecting \(G(x_0)\) and \(G(x_1)\) will intersect the function \(f(x) = \frac{\log(|\log(x)|)}{|\log(x)|^{\gamma}}\) again, in order for the point \(x_2\) to be well-defined.

To see this, we observe that the equation of the line \(L_{(G(x_0), 0)}(x)\) connecting \(G(x_0)\) and zero is:
$$  L_{(G(x_0), 0)}(x) = \frac{1}{x_0} \frac{\log(|\log(x_0)|)}{|\log(x_0)|^{\gamma}} x .$$
Thus, in order to show that the slope of the line \(L_{(G(x_0), G(x_1))}\) connecting \(G(x_0)\) and \(G(x_1)\) is less than the slope of the line \(L_{(G(x_0), 0)}\) connecting \(G(x_0)\) and zero, it suffices to show that \(L_{(G(x_0), 0)}(x_1) < G(x_1)\). Hence, we have

$$L_{(G(x_0),0)}(x_1)= \frac{1}{x_0}\frac{\log(|\log(x_0)|}{|\log(x_0)|^{\gamma}} x_1.$$

Then, by using the definition of $x_1$. Namely,   $x_1|\log(x_1)|^{1-\gamma}=x_0$, we have

$$L_{(G(x_0),0)}(x_1)\leq \frac{2\log(|\log(x_1|\log(x_1)|^{1-\gamma})|)}{|\log(x_1)|} < \frac{1}{|\log(x_1)|^{\gamma}}=G(x_1), $$

and thus Property 1 is satisfied.

\item  \textbf{Property 2):} The slope of the line \(L_{(G(x_0), G(x_1))}\) connecting \(G(x_0)\) and \(G(x_1)\) is also less than the slope of the line \(L_{(G(x_1), 0)}\) connecting \(G(x_1)\) and zero.

We need this property in order for the function \(G(x)\) to be concave and vanish at zero. If the slope of \(L_{(G(x_0), G(x_1))}\) were greater than the slope of \(L_{(G(x_1), 0)}\), then \(G(x)\) would have to either break concavity or vanish before \(x = 0\).

To show this, We observe that the line connecting $G(x_1)$ and zero is:

$$L_{(G(x_1),0)}(x)=\frac{1}{x_1}\frac{1}{|\log(x_1)|^{\gamma}}x \implies L'_{(G(x_1),0)}(x)  =\frac{1}{x_1}\frac{1}{|\log(x_1)|^{\gamma}}. $$

Computing the slope of the line connecting $G(x_0)$ and  $G(x_1)$ as follows:

$$G'(x_1)=\frac{G(x_0)-G(x_1)}{x_1|\log(x_1)|^{1-\gamma}-x_1}.$$

Then, using again that $x_1|\log(x_1)|^{1-\gamma}=x_0$, we have

$$
G'(x_1)< \frac{2}{x_1}\frac{\log(|\log(x_1|\log(x_1)|^{1-\gamma})|)}{|\log(x_1)
|} < \frac{1}{x_1}\frac{1}{|\log(x_1)|^{\gamma}} =L'_{(G(x_1),0)}(x).
$$

Thus, Property 2 is verified, and we are done. 

\end{itemize}

\end{proof}

\begin{corollary}\label{ScaledConst}

   There exist  a family of  continuous monotone concave functions $G(x)$ on $[0,1]$ such that  given  $\lambda >0$ we have

    $$
\limsup_{x\rightarrow 0} \frac{G(\lambda x \,G(x) \, |\log(x)|)}{G(x)}=\infty.
$$  
Quantitatively, we can find a sequence $x_n \rightarrow 0$ such that 
 $$
\frac{G(\lambda x_n \,G(x_n) | \log(x_n)| )}{G(x_n)}\geq    \frac{1}{2}\log(|\log(x_n)|)
$$

\end{corollary}

\begin{proof} This corollary follows from the previous proposition, Proposition \ref{Const}. We briefly show how to modify the previous construction in order for this corollary to follow. Take $0<\gamma<1$, and since \(\lambda > 0\) is fixed, we choose \(x_0\) small enough such that \(\frac{x_0}{\lambda} \ll 1\), and then we proceed as follows:

$$
G(x_0)=\frac{\log(|\log(x_0)|)}{|\log(x_0)|^{\gamma}}.
$$

Next, we choose $x_1$ such that $\lambda x_1|\log(x_1)|^{1-\gamma}=x_0$, and define

$$
G(x_1)=\frac{1}{|\log(x_1)|^{\gamma}}.
$$

As in the previous construction, we define \(x_2\) to be the intersection of the line connecting \(G(x_0)\) and \(G(x_1)\), \(L_{(G(x_0), G(x_1))}\), with the function \(f(x) = \frac{\log(|\log(x)|)}{|\log(x)|^{\gamma}}\), and define $$G(x_2)=\frac{\log(|\log(x_2)|)}{|\log(x_2)|^{\gamma}}.$$

Next, we define \(x_3\) such that \(\lambda x_3|\log(x_3)|^{1-\gamma} = x_2\), then we just repeat the procedure as in the previous proposition. The only difference is that for an even \(n\), \(x_{n+1}\) is defined as \(\lambda x_{n+1}|\log(x_{n+1})|^{1-\gamma} = x_{n}\). Then we see that

    $$
\limsup_{x\rightarrow 0} \frac{G(\lambda x \,G(x) |\log(x)
|)}{G(x)}=\infty.
$$

\end{proof}

\section{Loss of modulus of continuity}\label{Loss}
In this section, we prove our main result, showing that there exists a family of moduli of continuity for the 2D Euler equations that is not propagated.

We will use the family of  functions constructed in Section \ref{ConsInit} for the radial part of the initial data. Next, as mentioned in the sketch of proof in the introduction (Subsection \ref{sketch}), we follow the general strategy of Elgindi-Jeong \cite{EJ1}, and Jeong \cite{J} more specifically. To control the Euler dynamics given an initial particle, we first use Yudovich estimates (Lemma \ref{YudEstFlow}) to control the radial trajectory of the flow map and to approximately preserve the $\sin(2\theta)$  symmetry for a short time. This allows us to apply Lemma \ref{EJLemma} to obtain the leading order dynamics for the radial trajectory of the flow map. Using the structure of the initial data, we then show that the modulus of continuity is not propagated.

\begin{theorem}
There exist a family of solutions to the 2D Euler equations for which the modulus of continuity is lost. Namely there exist solutions to 2D Euler equations with initial  modulus of continuity  $\Omega_0(\rho)$ such that for $t>0$, we have

$$\frac{\Omega(t,\rho)}{\Omega_0(\rho)} \geq \frac{1}{2}\log(|\log(\rho)| )  $$

for $\rho>0$ arbitrarily small. 

\end{theorem}

\begin{remark}
The modulus of continuity of our initial data is, by construction, bounded from above by $\bar{\Omega}(\rho)= |\log(\rho)|^{-\gamma} \log(|\log(\rho)|)$,   
which is propagated, see Section \ref{PropEx}.
\end{remark}

\begin{proof}

We consider initial data of the following form:
\[
\omega_0(r,\theta) =  g(r) h(\theta),
\]
where \( h(\theta) \) is a smooth bump function with \( h(\theta) = 1 \) on \([\delta, \frac{\pi}{2}-\delta]\) with \(\delta \ll 1\),  and then extended on \([0, 2\pi]\) according to \(\sin(2\theta)\) symmetry.  \( g(r) \) is the continuous monotone concave function on \([0, 1 ]\), \( G(r) \), constructed in Proposition \ref{Const}, and then compactly extended on \((1, \infty)\). Thus by construction, we can always choose \(\rho\) arbitrarily small such that we have
\begin{equation}\label{InitialModulus}
\begin{split}
\Omega_0(\rho) = g(\rho) = \frac{1}{|\log(\rho)|^{\gamma}}
\end{split}
\end{equation}

The goal is to show that for small enough \( r_0 \) and \( \theta_0 \), we have the following estimate for the radial velocity of the flow map:
\[ 
\frac{d}{dt} \Phi^{-1}_{r}(t) \approx \Phi^{-1}_{r}(t) g(r_0) \, |\log(r_0)|.
\]

Now, before we start tracking the radial trajectory \( \Phi_{r}(t) \), by using Lemma \ref{EJLemma}, we need to use the following bounds in order to control the dynamics on the desired time interval \([0, \tau]\). Because the initial data is odd in both Cartesian variables, the origin is fixed for all time. In addition, we have
\begin{equation}\label{OddSymmetry}
\begin{split}
\omega(t,x_1,0)=\omega(t,0,x_2)=\omega(t,0,0)=0, \quad  \text{for all} \quad t \geq 0.
\end{split}
\end{equation}
Thus, using estimates \eqref{YudEst}, then for small enough time $\tau$, we have

\begin{equation}\label{radialYudEst}
\begin{split}
r_0^{1+ct}  \leq  \Phi_{r}(t)\leq r_0^{1-ct},  \quad \text{and}  \quad r_0^{1+ct}  \leq  \Phi^{-1}_{r}(t)\leq r_0^{1-ct}, 
\end{split}
\end{equation}
for all \( 0 \leq t \leq \tau \). Next, we would like to obtain upper bounds on the angular trajectory \(\Phi_{\theta}(t)\), which we will use when we apply Lemma \ref{EJLemma}. Using the fact that the angular component of the velocity field vanishes on the axes, and given the odd symmetry of the initial data in both Cartesian coordinates, we have the following bound on the angular component of the velocity field by applying Yudovich estimates Lemma \ref{YudEst}:
 $$
u^{\theta}(\Phi(t))\leq C  \Phi_{r}(t) \Phi_{\theta}(t) \, |\log( \Phi_{r}(t)  \Phi_{\theta}(t))|,
 $$
Thus, we have the following upper bounds on the angular trajectory (see Section 3.2 in the work of Jeong \cite{J} for more details):
\[ 
\frac{d}{dt} \Phi_{\theta}(t) \leq C \Phi_{\theta}(t) | \log( \Phi_{r}(t) \Phi_{\theta}(t))| 
\]
Thus for time small enough, on  $[0,\tau]$, we have 
$$
\frac{d}{dt}\Phi_{\theta}(t) \leq   C\Phi_{\theta}(t) \, |\log(r_0^{1+CT}  \Phi_{\theta}(t)) |,
$$
and hence, whenever $\theta_0=r_0^c$ for some $c>0$, we have 
\begin{equation}\label{angularYudEst}
\begin{split}
\Phi_{\theta}(t)  \leq   \theta_0^{1-Ct} .
\end{split}
\end{equation}
Thus, by taking \( \tau \) and \( \theta_0 \) small enough,   we are guaranteed that \( \Phi_{\theta}(t) \) stays small (say \( \Phi_{\theta}(t) < \frac{\pi}{16} \)).

Now given \eqref{radialYudEst} and \eqref{angularYudEst}, we  will start applying Lemma \ref{EJLemma} to obtain the desired  lower bounds on radial trajectory $\Phi^{-1}_{r}$. Recall from Lemma \ref{EJLemma}, we have 

   $$
\Big|u(r,\theta)-u(0)-\frac{1}{2\pi} \begin{pmatrix}
-\cos(\theta)\\
\sin(\theta)
\end{pmatrix} r I^s(r) -\frac{1}{2\pi} \begin{pmatrix}
\sin(\theta)\\
\cos(\theta)
\end{pmatrix} r I^c(r) \Big| \leq C  r |\omega|_{L^\infty},
 $$
where 
$$
I^s(r)=\int_r^{\infty} \int_0^{2\pi} \sin(2\theta)\frac{ \omega(s,\theta)}{s}  \, d \theta \, ds.
$$

Now since the initial we are considering is odd in both Cartesian variables, and it is propagates for all time. We have  $u(0)=0$ and $I^c(r)=0$. Thus we can write the velocity field as follows:

 $$
u(r,\theta)=\frac{1}{2\pi} \begin{pmatrix}
-\cos(\theta)\\
\sin(\theta)
\end{pmatrix} r I^s(r)  + R(r,\theta),  $$
 
 such that, by Lemma \ref{EJLemma}, we have 
\begin{equation}\label{RLipschits}
\begin{split}
  |R(r,\theta)| \leq  r |\omega|_{L^\infty}.
\end{split}
\end{equation}

Hence, we can write the flow of the radial trajectory as follows:    
$$
\frac{d} {dt}  \Phi_{r}(t)=-  \cos(2 \Phi_{\theta}(t)))  \Phi_{r}(t) I^s(\Phi_{r}(t))+  R_r(\Phi_{r},\Phi_{\theta})
$$

Thus, by \eqref{RLipschits} and  \eqref{angularYudEst}, we have

$$
\frac{d} {dt}  \Phi_{r}(t)\leq -  c \,   \Phi_{r}(t) I^s(\Phi_{r}(t)).
$$

As mentioned before, since the initial data is odd symmetric and it is propagated for all time, we have
$$
I^s(\Phi_{r}(t))=\int_{\Phi_{r}(t)}^{\infty} \int_0^{2\pi} \sin(2\theta)\frac{ \omega(t,s,\theta)}{s}  \, d \theta \, ds=4\int_{\Phi_{r}(t)}^{\infty} \int_0^{\frac{\pi}{2}} \sin(2\theta)\frac{ \omega(t,s,\theta)}{s}  \, d \theta \, ds.
$$

Thus, we obtain 
$$
I^s(\Phi_{r}(t)) \geq c\int_{\Phi_{r}(t)}^{1} \int_{\Phi_{\theta}(t)}^{\frac{\pi}{4}} \sin(2\theta)\frac{ \omega(t,s,\theta)}{s}  \, d \theta \, ds. 
$$

Now since the vorticity is being transported, by definition we have  
\begin{equation}\label{VorticityDefinition}
\begin{split}
\omega(t,s,\theta)= \omega_0(\Phi^{-1}_r(t,s), \Phi^{-1}_{\theta}(t,\theta))=  g(\Phi^{-1}_{r}(t,s)) h(\Phi^{-1}_{\theta}(t,\theta)).
\end{split}
\end{equation}
Hence, we obtain 
$$
I^s(\Phi_{r}(t)) \geq c\int_{\Phi_{r}(t)}^{1} \int_{\Phi_{\theta}(t)}^{\frac{\pi}{4}} \sin(2\theta)\frac{  g(\Phi^{-1}_{r}(t,s)) h(\Phi^{-1}_{\theta}(t,\theta))}{s}  \, d \theta \, ds.  
$$

Now on $[0,1]$, by construction,  $g(r)$ is monotonically increasing. Thus on $[\Phi_{r}(t),1]$  we have $g(\Phi^{-1}_{r}(t,s)) \geq g(\Phi^{-1}_{r}(t,\Phi_{r}(t)))= g(r_0)  $ for all $s \in [\Phi_{r}(t),1]$. Hence, 

$$
I^s(\Phi_{r}(t)) \geq c g(r_0) \int_{\Phi_{r}(t)}^{1} \int_{\Phi_{\theta}(t)}^{\frac{\pi}{4}} \frac{ \sin(2\theta) h(\Phi^{-1}_{\theta}(t,\theta))}{s}  \, d \theta \, ds.    
$$

  In addition,   since we chose $h(\theta)$ to be a smooth bump function with $h(\theta)=1$ on $[\delta,\frac{\pi}{2}-\delta]$,  for $\delta \ll1$,  using estimates  \eqref{angularYudEst}, we are guaranteed  that $\Phi_{\theta}(t)  \leq \frac{\pi}{16}$ on $[0,\tau]$. Hence, the angular integral stays uniformly bounded.     Thus, we have

$$
I^s(\Phi_{r}(t)) \geq c g(r_0) |\log( \Phi_{r}(t))) |. 
$$

  Hence, on our time interval, we have

$$
\frac{d}{dt}\Phi_{r}(t) \leq   - c \Phi_{r}(t) g(r_0)  |\log( \Phi_{r}(t))) |. 
$$
Then,  using \eqref{radialYudEst}, we obtain  
$$
\frac{d}{dt}\Phi_{r}(t) \leq   - c \Phi_{r}(t) g(r_0)  |\log(r_0)|.   
$$
Hence,
\begin{equation}\label{FlowAngle}
\begin{split}
\Phi_{r}(t) \leq     r_0 e^{-cg(r_0)   \, |\log(r_0)| t}.  
\end{split}
\end{equation}
From \eqref{FlowAngle}, and because of the monotonicity of the initial data, we can estimate the trajectory of the inverse of the flow map as follows:
\begin{equation}\label{InverFlowAngle}
\begin{split}
\Phi^{-1}_{r}(t) \geq     r_0 e^{cg(r_0) \, |\log(r_0)| t}    \geq r_0(1+cg(r_0)  |\log(r_0)| t).
\end{split}
\end{equation}
Thus, from  estimate \eqref{InverFlowAngle},  we can obtain a lower bound on the vorticity:
\begin{equation*}
\begin{split}
\omega(t,r_0,\theta_0)=g(\Phi^{-1}_r(t))h\big(\Phi^{-1}_{\theta}(t)) \geq g(c   r_0 g(r_0) \,  |\log(r_0)|  t ) \, h\big(\Phi^{-1}_{\theta}(t)), 
\end{split}
\end{equation*}
and recall by the choice of initial data $h(\theta)$, using \eqref{angularYudEst}, we can choose $\theta_0$ such that we have $$h(\Phi^{-1}_{\theta}(t))=1 \quad \text{for}  \quad 0\leq t\leq \tau.$$ 
Thus,  we have 
\begin{equation}\label{VorticityLowerBnd}
\begin{split}
\omega(t,r_0,\theta_0) \geq g(c   r_0 g(r_0) \, |\log(r_0)|   t).
\end{split}
\end{equation}

Now with \eqref{VorticityLowerBnd}, we are ready to show the loss of modulus of continuity. 
Recall that the modulus of continuity is defined as follows:
$$
\Omega(t,\rho)=\sup_{|x-y|\leq \rho}|\omega(t,x)-\omega(t,y)|. 
$$
From \eqref{InitialModulus}, we have 
$$\Omega_0(\rho)=g(\rho)=\frac{1}{|\log(\rho)|^{\gamma }}.$$
The goal is to obtain a lower bound on $\Omega(t,\rho)$ using \eqref{VorticityLowerBnd}.  In polar coordinates, we take $x=(r_0,\theta_0)=(\rho,\rho^c)$, and  $y=(\rho,0)$, for $0<c<1$. As  before, because of odd symmetry of initial data \eqref{OddSymmetry}, we have  $\omega(t,y)=0$ for all $t$. Thus, with the choice  $x=(r_0,\theta_0)=(\rho,\rho^c)$, we have 
$$
\Omega(t,\rho)\geq \omega(t,r_0,\theta_0)=\omega(t,\rho,\rho^{c}).
$$

Hence,  from  \eqref{VorticityLowerBnd},  we obtain

$$
\frac{\Omega(t,\rho)}{\Omega_0(\rho)} \geq  \frac{g(c   \rho g(\rho)  \, |\log(\rho) |  t)}{g(\rho)}. 
$$

Now by Corollary \ref{ScaledConst}, we can choose the scaling parameter $\lambda$ such that $\lambda=c\tau$ and obtain:

$$
\frac{\Omega(t,\rho)}{\Omega_0(\rho)} \geq \frac{1}{2} \log(|\log(\rho)|),
$$
for $0<t\leq \tau$, and this gives the proof.   

\end{proof}

\section*{Acknowledgements}
I would like to thank my advisor Professor Tarek M. Elgindi for suggesting this problem, and I am grateful for all his important suggestions and feedback while working on this project. I also would like to thank Dr. Ayman R. Said for his detailed comments and important remarks on the paper, especially for his remarks on using Lemma 2.3.

\vskip 0.3 cm

Department of Mathematics, Duke University, Durham, NC 27708, USA

\textit{Email address:} karimridamoh.shikh.khalil@duke.edu

\end{document}